\documentclass{amsart}

\usepackage{graphicx}
\usepackage[usenames,dvipsnames]{color}

\theoremstyle{plain}
\newtheorem{theorem}{Theorem}[section]
\newtheorem{lemma}[theorem]{Lemma}
\newtheorem{corollary}[theorem]{Corollary}

\theoremstyle{definition}
\newtheorem{definition}[theorem]{Definition}

\theoremstyle{remark}

\newtheorem*{acknowledgements}{Acknowledgements}

\newcommand{\del}{\partial}
\newcommand{\R}{\mathbb{R}}
\newcommand{\Z}{\mathbb{Z}}

\newcommand{\C}{\mathbb{C}}

\newcommand{\half}{\frac{1}{2}}
\begin{document}

\title{Prism complexes}
\author[Kalelkar]{Tejas Kalelkar}
\address{Mathematics Department, Indian Institute of Science Education and Research, Pune 411008, India}
\email{tejas@iiserpune.ac.in}

\author[Nair]{Ramya Nair}
\address{Mathematics Department, Indian Institute of Science Education and Research, Pune 411008, India}
\email{nair.ramya@students.iiserpune.ac.in}

\date{\today}



\begin{abstract} 
A prism is the product space $\Delta \times I$ where $\Delta$ is a 2-simplex and $I$ is a closed interval. As an analogue of simplicial complexes, we introduce prism complexes and show that every compact $3$-manifold has a prism complex structure. We call a prism complex special if each interior horizontal edge lies in four prisms, each boundary horizontal edge lies in two prisms and no horizontal face lies on the boundary. We give a criteria for existence of horizontal surfaces in (possibly non-orientable) Seifert fiber spaces. Using this we show that a compact 3-manifold admits a special prism complex structure if and only if it is a Seifert fiber space with non-empty boundary, a Seifert fiber space with a non-empty collection of surfaces in its exceptional set or a closed Seifert fiber space with Euler number zero. So in particular, a compact $3$-manifold with boundary is a Seifert fiber space if and only if it has a special prism complex structure.
\end{abstract}

\maketitle

\section{Introduction}\label{intro}
A closed $3$-manifold $M$ is called irreducible if every embedded $2$-sphere in $M$ bounds a $3$-ball. If $M$ is reducible then there exist finitely many disjointly embedded $2$-spheres in $M$ such that after cutting $M$ along these spheres and capping off the spherical boundaries with $3$-balls, the closed manifolds obtained are either irreducible or $S^2 \times S^1$. These capped off pieces can be further decomposed along a canonical collection of disjointly embedded tori into compact manifolds which have three possibilities: they are either finitely covered by torus bundles, they are Seifert fiber spaces with boundary or they have interiors which admit a complete hyperbolic metric. Cube complexes give a nice discrete structure to study hyperbolic manifolds. Gromov gave a combinatorial criterion on cube complexes which ensures that they have a CAT(0) metric. In this article we introduce prism complexes as discrete structures to study Seifert fiber spaces. In particular we give a combinatorial criterion on prism complexes to ensure that the underlying manifold is in fact a Seifert fiber space.

A \emph{Seifert fiber space} is a compact connected $3$-manifold which admits a foliation by circles. Let $D=\{(x, y) \in \R^2 : x^2 + y^2 \leq 1\}$ and let $D^+ = \{(x, y) \in D: x\geq 0\}$. A \emph{regular} fiber in a Seifert fiber space has a fibered neighbourhood fiber preservingly homeomorphic to a fibered solid torus $D \times S^1$ or half solid torus $D^+ \times S^1$ trivially foliated by the leaves $x \times S^1$. Fibers which are not regular are called \emph{exceptional}. Exceptional fibers in the interior of the manifold that are isolated have fibered solid torus neighbourhoods fiber preservingly homeomorphic to the flow of the mapping torus of $D$ via a monodromy that is a rational rotation. Exceptional fibers in the interior of the manifold which are not isolated have fibered solid Klein bottle neighbourhoods fiber preservingly homeomorphic to the flow of the mapping torus of $D$ via the reflection monodromy $r(x,y) =(x, -y)$. Exceptional fibers on the boundary of the manifold are not isolated. They have a fibered half solid Klein bottle neighbourhood fiber preservingly homeomorphic to the flow of the mapping torus of $D^+$ via the reflection monodromy $r(x, y)=(x, -y)$. All of these fiber-preserving homeomorphisms take the exceptional fiber to the fiber above $0\in D$ (or $0 \in D^+$). The connected components of the set of exceptional fibers are therefore circles, annuli, tori or Klein bottles. We denote by $SE(M)$ the collection of annuli, tori and Klein bottle components of the exceptional set. When $M$ is orientable, $SE(M)=\emptyset$. When $M$ is closed we denote by $e(M)$ the Euler number of the manifold. A good reference for Seifert fiber spaces is the survey paper by Scott\cite{Sco} and the preprint of a book by Hatcher\cite{Hat}.

A \emph{prism} is the product space $\Delta \times I$ where $\Delta$ is a 2-simplex and $I$ a closed interval. We call the edges of $\del \Delta \times \del I$ \emph{horizontal} and the rest of the edges of the prism we call \emph{vertical}. Similarly, we call the faces in $\Delta \times \del I$ \emph{horizontal} and faces in $\del \Delta \times I$ \emph{vertical}. A \emph{prism complex} is a 3-dimensional cell complex in which each cell is a prism, the attaching maps are combinatorial isomorphisms and furthermore, horizontal edges are identified only with horizontal edges. We call a prism complex \emph{special} if each horizontal edge in the interior of the complex lies in four prisms, each boundary horizontal edge lies in two prisms and no horizontal face lies on the boundary of the complex. 
 
We prove in this paper that a special prism complex can be thought of as a discrete version of the local fibration of a Seifert fiber space:
 \begin{theorem}\label{mainthm} 
Every compact $3$-manifold $M$ admits a prism complex structure. Moreover, it admits a special prism complex structure if and only if it is a Seifert fiber space with $\del M \neq \emptyset$ or $SE(M)\neq \emptyset$ or $e(M)=0$.
 \end{theorem}

So in particular, if $M$ is a compact 3-manifold with boundary, then it admits a special prism complex structure if and only if it is a Seifert fiber space.\\

Most of the literature on Seifert fiber spaces deals only with oriented Seifert fiber spaces with the corresponding results for non-oriented spaces being folklore. To prove our result for all Seifert fiber spaces, we explicitly give a general criteria for existence of horizontal surfaces. A \emph{horizontal} surface in a Seifert fiber space $M$ is an embedded surface that is transverse to all the circle fibers of $M$.

\begin{theorem}\label{SFSThm}
Let $M$ be a Seifert fiber space.
\begin{enumerate}
\item When $\partial M\neq \emptyset$ or $SE(M)\neq \emptyset$ then horizontal surfaces exist in $M$.
\item When $\partial M=\emptyset$ and $SE(M)=\emptyset$ then horizontal surfaces exist if and only if $e(M)=0$.
\end{enumerate}
\end{theorem}

 \section{Seifert fiber spaces} 
This section deals with the construction of Seifert fiber spaces and a proof of Theorem \ref{SFSThm}. A complete combinatorial description for Seifert fiber spaces, which includes the non-oriented spaces, is explained in detail by Cattabriga et al\cite{CMMN}:
  
\begin{theorem}[Theorem A of \cite{CMMN}]\label{SFSClassification}
Every Seifert fiber space is uniquely determined, up to fiber-preserving homeomorphism, by the normalised set of parameters
$\{b;(\epsilon,g,(t,k));(h_1,...,h_{m_+} \mid k_1,...,k_{m_-});((p_1,q_1),...,(p_r,q_r))\}$.
\end{theorem}

See Section 2 of \cite{CMMN} for a description of the parameters in the above theorem and for an explicit construction of a Seifert fiber space with the above parameters. We give below an outline of the construction:\\
\begin{figure}
\centering
\def\svgwidth{0.6\columnwidth}
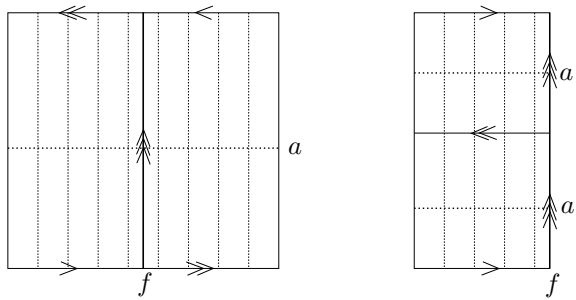
\caption{Different representations of the same fibered Mobius strip $N$ with exceptional fiber $f$ and a horizontal arc $a$}\label{Mob}
\end{figure}

\emph{Construction of Seifert fiber space $M$ with given parameters:}
Let $B^*$ be a compact connected surface of genus $g$ with $m_+ + m_- + (r+1) + t$ boundary components.  $B^*$ is orientable if $\epsilon=o$, $o_1$ or $o_2$ and non-orientable otherwise (i.e. if $\epsilon=n, n_1, n_2, n_3, n_4$). Consider $B^*$ as a disk $D^*$ with disjoint arcs $\sigma_i, \sigma'_i$ on the boundary identified. As $D^*$ is contractible, a circle bundle over $D^*$ is necessarily the trivial bundle $D^* \times S^1$. Any circle bundle $p: M^* \to B^*$ is then obtained from $D^* \times S^1$ by identifying the disjoint annuli $p^{-1}(\sigma_i)$ and $p^{-1}(\sigma'_i)$ fiber preservingly. As any homeomorphism of a circle is isotopic to the identity or the reflection map so we may assume that these annuli are identified by the identity or reflection map along the fibers. This pairwise identification is determined by the symbols for $\epsilon$ and such that $M^*$ ends up with $(r+1) + (t-k) + m_+$ torus boundary components and $k+m_{-}$ Klein bottle boundary components. As both the identity and the reflection map on $S^1$ have a fixed point, so we can identify $B^*$ with a fixed section of this circle bundle. We now obtain $M$ from $M^*$ via the following steps:

\emph{Step 1: }Let $T_i$ denote the torus boundary components of $M^*$. On each such boundary component define the meridian $\mu_i$ as the curve $T_i \cap \del B^*$ and choose a regular boundary fiber of $T_i$ as the longitude $\lambda_i$.  Let $V_i$ be solid tori. Define the meridian on $\del V_i$ as the unique curve (up to isotopy) that bound a disk in $V_i$. By a Dehn filling of $T_i$ by $V_i$ along the slope $q_i/p_i$ we mean the attachment of $V_i$ to $M^*$ via a homeomorphism from $\del V_i$ to $T_i$ that sends the meridian of $V_i$ to the curve $p_i \mu_i + q_i \lambda_i$. Put $(p_{r+1}, q_{r+1})=(1,b)$.  As the first step in our construction, we Dehn fill the first $r+1$ torus boundary components $T_i$ with solid tori $V_i$ along the given slopes $q_i/p_i$. Let $M'$ be the manifold thus obtained. 

\emph{Step 2: }Let $N=I \times I/(x, 0) \sim (1-x, 1)$ be a mobius strip foliated by the circles $(x \times I) \cup ((1-x) \times I)/\sim$ as in Figure \ref{Mob}(i). Let $\phi_i: S^1 \times \del N \to T_i$ be the homeomorphism sending $t \times \del N$ to $\mu_i (t) \times S^1$ in $T_i$. It is helpful to consider the model of $N$ with the boundary on one side as in Figure \ref{Mob}(ii) (which can be obtained from the model in Figure \ref{Mob}(i) by cutting along the fiber $f$, flipping one of the pieces and reattaching along the segment with the double arrows). As a second step in our construction we attach $S^1 \times N$ to the next $(t-k)$ torus boundary components $T_i$ via the attaching map  $\phi_i$. We call this process capping off $T_i$ via $S^1 \times N$.

\emph{Step 3: }In each torus $T_i$ of the remaining $m_+$ torus boundary components let $\{\gamma_j\}_{j=1}^{h_i}$ be $h_i$ many disjoint arcs in $\mu_i$. Let $\psi_{(i,j)}: I \times \del N \to p^{-1}(\gamma_j) \subset T_i$ be the homeomorphism sending $t \times \del N$ to $\gamma_j(t)\times S^1$ in $T_i$. In this step we attach a copy of $(I \times N)$ to each $p^{-1}(\gamma_j) \subset T_i$ via the attaching homeomorphism $\psi_{(i, j)}$. So if $h_i>0$, then the torus boundary $T_i$ of $M'$ is replaced by $h_i$ many Klein bottle boundary components. 

\emph{Step 4} Let $K_i$ denote the Klein bottle boundary components of $M'$. We express $K_i$ as a twisted product $S^1 \widetilde{\times} S^1 = I \times S^1/\sim$ where $(0, z) \sim (1, -z)$ with $B^* \cap K_i$ the curve $\mu_i(t)=(t, e^{\pi i t})$. Let $S^1 \widetilde{\times} N$ denote the twisted product $I \times N/\sim$ where $(0, (x, y))\sim(1, (1-x, 1-y))$ with $N=I \times I/(x, 0) \sim (1-x, 1)$. Let $\phi'_i: S^1 \widetilde{\times} \del N \to K_i$ be the homeomorphism sending $t \times \del N$ to the fiber above $\mu_i(t)$ in $K_i$.  In this step, we cap off the first $k$ Klein bottle boundary components of $M'$ by $S^1 \widetilde{\times} N$ via the attaching map $\phi'_i$.

\emph{Step 5} In each Klein bottle $K_i$ of the remaining $m_-$ Klein bottle boundary components let $\{\gamma'_j\}_{j=1}^{k_i}$ be $k_i$ many disjoint arcs in $\mu_i$. Let $\psi'_{(i, j)}: I \times \del N \to p^{-1}(\gamma'_j) \subset K_i$ be the homeomorphism sending $t \times \del N$ to the fiber above $\gamma'_j(t)$ in $K_i$. In this final step, we attach a copy of $(I \times N)$ to each $p^{-1}(\gamma'_j) \subset K_i$ via the attaching homeomorphism $\psi'_{(i, j)}$ to obtain $M$. If $k_i>0$, then the Klein bottle boundary $K_i$ of $M'$ is replaced by $k_i$ many Klein bottle boundary components.


The manifold $M$ is closed if and only if $m_+ + m_- = 0$ and oriented if and only if $\epsilon =o_1$ or $n_2$, $m_-=t=0$ and $h_i=0$ for all $i=1,...,m_+$. For a closed Seifert fiber space, the Euler number of the fibering is given by $e(M)=\sum_{i=1}^r \frac{q_i}{p_i} + b$.

\begin{figure}
\centering
\def\svgwidth{0.7\columnwidth}
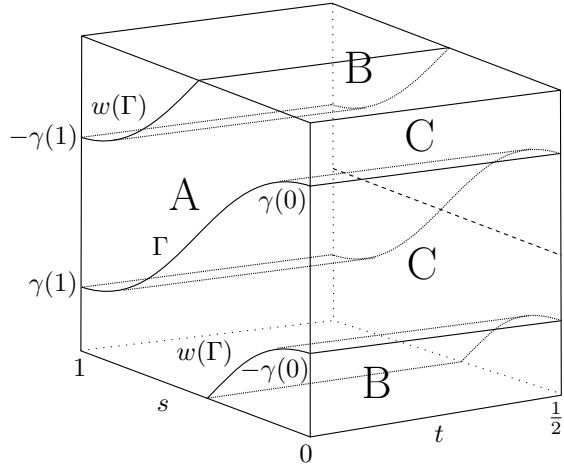
\caption{$M=I \times N$ depicted as a cube $[0,1]\times [0,\frac{1}{2}] \times [0,1]$ with the top and bottom faces $B$ identified and the two halves $C$ of a vertical faces identified. On the annulus $A$ is shown the curves $\Gamma$ and $w(\Gamma)$, and in the interior of the cube is the rectangle $R_\Gamma$.}\label{HorFig}
\end{figure}

Let $N$ be a Mobius strip foliated by circles with one exceptional fiber $f$ as in Figure \ref{Mob}. We give a sufficient condition below on when a set of curves on the boundary of $S^1 \times N$ or $S^1 \widetilde{\times} N$ bound a  horizontal surface. Recall that a horizontal surface is an embedded surface that is transversal to all the fibers of $M$.  Keep Figure \ref{HorFig} as reference for Lemma \ref{extension}.


\begin{lemma}\label{extension}
Let $N=[0,\half] \times S^1/\sim$ with $(\half,z) \sim (\half, -z)$ be a Mobius strip foliated by the circles $t\times S^1$ with one exceptional fiber $\half \times S^1$ as in Figure \ref{Mob}(ii). Let $M = I \times N$ be foliated by the fibers $s \times t \times S^1$ for $s \in I$, $t\in[0,\half]$, with an exceptional annulus $I \times \half \times S^1$. Let $\Gamma$ be a properly embedded arc in $A=I \times \del N = I \times 0 \times S^1$ which intersects each fiber $s \times \del N= s\times 0 \times S^1$ transversely. $\Gamma$ has a parametrisation $\Gamma(s)=(s, 0, \gamma(s))$, for some arc $\gamma: I \to S^1$. Let $\omega: A \to A$ be the map $(s,0, z) \to (s, 0, -z)$. Then there exists a horizontal rectangle $R_\Gamma$ in $M$ such that $R_\Gamma \cap A = \Gamma \cup \omega(\Gamma)$ and for $j=0, 1$, $R_\Gamma \cap (j\times N) =j \times [0,\half] \times (\gamma(j) \cup -\gamma(j))/\sim$. Furthermore if $\Gamma'$ is another properly embedded arc in $A$ disjoint from $(\Gamma \cup \omega(\Gamma))$ then $R_\Gamma$ is disjoint from $R_{\Gamma'}$. 
\end{lemma}
\begin{proof}
As the arc $\Gamma$ is transverse to the foliation $s \times 0 \times S^1$ of $A=I \times 0 \times S^1$ so it intersects each fiber exactly once. We can therefore parametrise the arc as $\Gamma(s) = (s, 0, \gamma(s))$ for some arc $\gamma: I \to S^1$. Define the embedding $r_\Gamma: I \times I \to M$ as follows:

$$r_\Gamma(s, t) = \left\{\begin{array}{ll}
(s, t, \gamma(s)) & \mbox{if $t\in[0, \half]$}\\
(s, 1-t, -\gamma(s)) & \mbox{if $t\in [\half, 1]$}
\end{array}\right.$$

Let $R_\Gamma$ be the image of $r_\Gamma$. Then $R_\Gamma \cap A$ is  $r_\Gamma(I \times 0)\cup r_{\Gamma}(I \times 1) = \Gamma \cup \omega(\Gamma)$. And for $j=0, 1$, $R_\Gamma \cap (j \times N) = r_\Gamma(j \times I) = (j \times I \times \gamma(j)) \cup (j \times I \times -\gamma(j))$. If $\Gamma'(s)=(s, 0, \gamma'(s))$ is disjoint from $\Gamma \cup \omega(\Gamma)$ then for any $s\in I$, $\gamma'(s) \neq \pm \gamma(s)$. Hence the images of $r_\Gamma$ and $r_{\Gamma'}$ are disjoint as required.



\end{proof}

\begin{lemma}\label{extension2}
Let $M_1 = S^1 \times N$ and let $M_2 = S^1 \widetilde{\times} N = I \times [0, \half] \times S^1/\sim$ with $(0, t, z) \sim (1, t, \bar{z})$ and $(s, \half, z) \sim (s, \half, -z)$, for $s\in I$, $t\in[0, \half]$, $z\in S^1$. Take their foliation by the circles $s \times t \times S^1$ which has an exceptional torus $S^1 \times \half \times S^1$ in $M_1$ and exceptional Klein bottle $I \times \half \times S^1/\sim$ in $M_2$. For $i=1, 2$, let $\Lambda_i$ be a simple closed curve on $\del M_i$ which intersects each fiber $s \times \del N$ transversely. Let $\omega: \del M_i \to \del M_i$ be the map $(s, z) \to (s, -z)$ (as $-\bar{z}=\overline{(-z)}$ so $\omega$ is well-defined on $\del M_2$). Then there exists a horizontal annulus or horizontal Mobius strip $R_{\Lambda_i}$ in $M_i$ such that $R_{\Lambda_i} \cap \del M_i = \Lambda_i \cup \omega(\Lambda_i)$. Furthermore if $\Lambda'_i$ is another simple closed curve in $\del M_i$ disjoint from  $\Lambda_i \cup \omega(\Lambda_i)$ then $R_{\Lambda_i}$ is disjoint from $R_{\Lambda'_i}$.
\end{lemma}
\begin{proof}
Let $M= I \times N= I \times [0, \half] \times S^1/(s, \half, z) \sim (s, \half, -z)$ as in Figure \ref{HorFig}. Let $M_1 = S^1 \times N = I \times [0, \half] \times S^1/\sim$ where $(s, \half, z) \sim (s, \half, -z)$ and $(0, t, z) \sim (1, t, z)$. There exists a quotient map $q: M \to M_1$ such that $q(1, t, z)=q(0, t, z)$. In particular, $q(1, 0, z)=q(0, 0, w)$ if and only if $z=w$.

The preimage $q^{-1}(\Lambda_1)$ of $\Lambda_1$ is a disjoint collection of  properly embedded arcs $\{\Gamma_i\}_{i=0}^k$ in $A=I \times 0 \times S^1$ transverse to the fibers $s \times 0 \times S^1$. There exists a parametrisation $\Gamma_i(s)=(s, 0, \gamma_i(s))$ with $\gamma_i:I \to S^1$. Furthermore, $q(\Gamma_i(1))=q(1, 0, \gamma_i(1))=q(0, 0, \gamma_{i+1}(0))=\Gamma_{i+1}(0)$ (taking $i$ modulo $k+1$). So in particular, $\gamma_i(1)=\gamma_{i+1}(0)$ in $S^1$. And therefore $q(1, t, \gamma_i(1))=q(0, t, \gamma_{i+1}(0))$ ($i$ mod $k+1$).

 By Lemma \ref{extension} there exist horizontal rectangles $R_i$ in $M$ such that $R_i \cap A=\Gamma_i \cup \omega(\Gamma_i)$ and for $j=0, 1$, $R_i \cap (j \times N) = j \times I \times (\gamma_i(j) \cup -\gamma_i(j))$. Let $R_{\Lambda_1} = \cup_{i=0}^k q(R_i)$ in $M_1$. As $q(A) = \del M_1$ and $q(\cup_{i=0}^k (\Gamma_i \cup \omega(\Gamma_i))) = \Lambda_1\cup \omega(\Lambda_1)$ so $R_{\Lambda_1} \cap \del M_1 = \Lambda_1 \cup \omega(\Lambda_1)$. As $q(1, t, \gamma_i(1))=q(0, t, \gamma_{i+1}(0))$, so $q(R_i \cap 1 \times N)=q(R_{i+1} \cap 0 \times N)$ ($i$ mod $k+1$). Hence the edges of $q(R_i)$ that lie on the boundary of $M_1$ match up to give the horizontal annulus $R_{\Lambda_1}$. 
\newline

Similarly, $M_2 = S^1 \widetilde{\times} N = I \times [0, \half] \times S^1/\sim$ where $(s, \half, z) \sim (s, \half, -z)$ and $(0, t, z) \sim (1, t, \bar{z})$. Let $q: M \to M_2$ be the quotient map so that $q(0, t, z) = q(1,t, \bar{z})$. In particular, $q(0, 0, z)=q(1, 0, w)$ if and only if $z=\bar{w}$. Let $q^{-1}(\Lambda_2)$ be a disjoint collection of properly embedded arcs $\{\Gamma_i\}_{i=0}^k$ in $A$ with a parametrisation $\Gamma_i(s)=(s, 0, \gamma_i(s))$. Furthermore, $q(\Gamma_i(1))=q(1, 0, \gamma_i(1))=q(0, 0, \gamma_{i+1}(0))=q(\Gamma_{i+1}(0))$ (taking $i$ modulo $k+1$). So in particular, $\gamma_i(1)=\overline{\gamma_{i+1}(0)}$ in $S^1$. And therefore $q(1, t, \gamma_i(1))=q(0, t, \gamma_{i+1}(0))$ ($i$ mod $k+1$). Finally using Lemma \ref{extension} as above, there exists a horizontal annulus or horizontal Mobius strip $R_{\Lambda_2}$ in $M_2$ such that $\del R_{\Lambda_2}=\Lambda_2 \cup \omega(\Lambda_2)$ (note that $R_{\Lambda_2}$ is a Mobius strip when $\omega(\Lambda_2)=\Lambda_2$). 
\end{proof}

\begin{lemma}\label{wellspaced}
Let $n$ be an even positive number and let $\theta_0=2\pi/n$. Let $w=e^{i\theta_0/2}$ and let $P(n)=\{w e^{mi\theta_0}: m \in \Z\}=\{w, w e^{i\theta_0}, w e^{2i\theta_0}, ... ,w e^{(n-1)i\theta_0}\}$. Let $\rho:S^1 \to S^1$ be the reflection map $z \to \bar{z}$ and let $\omega: S^1 \to S^1$ be the antipodal map $z \to -z$. Then $\rho(P(n))=P(n)$ and $\omega(P(n))=P(n)$.
\end{lemma}
\begin{proof}
As $e^{ni\theta_0}=1$, so points in $P(n)$ are of the form $we^{mi\theta_0} =e^{(m+1/2)i\theta_0}$ for $m\in \Z$. In particular $\rho(e^{(m+1/2)i\theta_0}) = e^{-(m+1/2)i\theta_0}=e^{((-m-1)+1/2)i\theta_0}$, so $\rho(P(n))=P(n)$. And $m\theta_0 + \pi =m\theta_0 + n\theta_0/2 = (m+n/2)\theta_0$. So $-we^{mi \theta_0} = we^{mi\theta_0 + i\pi}=we^{(m+n/2)i \theta_0} \in P(n)$ as $n$ is even. Therefore $\omega(P(n))=P(n)$.
\end{proof}
\begin{lemma}\label{welltwisted}
Let $0<q<p$ be coprime integers, and let $n=kp$ be an even number. Let $T$ be a torus $S^1 \times S^1\subset \C \times \C$ with meridian $\mu=S^1 \times 1$ and longitude $\lambda=1\times S^1$. Fix a point $z_0 \in S^1$ different from $1$. Let $\alpha:I \to S^1$ and $\beta: I \to S^1$ be the arcs in anti-clockwise direction from $1$ to $z_0$ and from $z_0$ to $1$ respectively. There exists a set of $k$ pairwise disjoint curves $\Lambda$ of slope $q/p$ such that $\Lambda \cap (\beta(s) \times S^1) = \beta(s) \times P(n)$ and $\Lambda \cap (\alpha(s) \times S^1)=\alpha(s) \times e^{2\pi qsi/p}P(n)$.
\end{lemma}
\begin{proof}
Let $A(s)=\alpha(s) \times e^{2\pi q si/p} P(n)$ and let $B(s)=\beta(s) \times P(n)$. Let $\theta_0=2\pi/n$ as in Lemma \ref{wellspaced}. Then $A(0)=1 \times P(n)=B(1)$ and $A(1)=z_0 \times e^{2\pi i q/p}  P(n)=z_0 \times e^{kq i \theta_0} P(n)= z_0 \times P(n) = B(0)$. So $\Lambda=A \cup B$ is a union of pairwise disjoint curves in $T$. 

Both $\lambda$ and $\mu$ intersect $\Lambda$ transversely with the same sign at every intersection, so the slope of a curve in $\Lambda$ is given by taking the ratio $|\Lambda \cap \mu|/|\Lambda \cap \lambda|$. To see that the slope of these curves is $q/p$ we shall collapse $\beta\times S^1$ to $1\times S^1$ so that the annulus $\alpha \times S^1$ becomes a torus $T'$ and $\Lambda$ goes to curves of slope $q/p$ in $T'$.

Let $z_0=e^{i\varphi}$ and let $f:T \to T'$ be the quotient map defined as follows:
$$f(z, w)=\left\{ \begin{array}{ll}
(z^{2\pi/\varphi}, w) & \mbox{if } z \in \alpha \\
(1, w) & \mbox{if } z\in \beta
\end{array}\right.
$$
As $f(A(0))=f(A(1))=1 \times P(n)$ so $f(A)$ is a pairwise disjoint set of curves in $T'$ parametrised by $s\to (e^{i(2\pi s)}, e^{i(2\pi s)q/p}P(n))$ for $s\in I$. Their lifts in $\R^2$ are parallel straight lines with slope $q/p$, so the curves in $f(A)$ have slope $q/p$ as required. Let $\lambda'=f(\lambda)$ and let $\mu'=f(\mu)$. Then the slope of $f(\Lambda)=f(A)$ is given by $|f(\Lambda) \cap \mu'|/|f(\Lambda) \cap \lambda'| = q/p$. 

Note that $f$ takes both $\lambda=1 \times S^1$ and $z_0 \times S^1$ to $1\times S^1$ homeomorphically. Furthermore, $\Lambda \cap \lambda = \Lambda \cap (z_0 \times S^1) = P(n)$. So $|f(\Lambda) \cap f(\lambda)|=|f(\Lambda \cap \lambda)|=|\Lambda \cap \lambda|$. As $\mu$ is disjoint from $B$, so $\Lambda \cap \mu = A \cap \mu \subset \alpha(0,1) \times S^1$. As $f$ restricted to $\alpha(0,1) \times S^1$ is a homeomorphism onto its image and $f(\Lambda) = f(A)$ so $|f(\Lambda) \cap f(\mu)|=|f(A) \cap f(\alpha(0, 1) \times 0)|=|A \cap (\alpha(0, 1) \times 0)|=|A \cap \mu|=|\Lambda \cap \mu|$. Therefore the slope of curves in $\Lambda$, is $|\Lambda \cap \mu|/|\Lambda \cap \lambda| = q/p$ as required. Also as $|\Lambda \cap \lambda|=|P(n)|=n$ so there are $n/p=k$ curves in $\Lambda$ as required.




\end{proof}

The criteria for existence of horizontal surfaces in orientable Seifert fiber spaces is well known.  We extend this criteria to all Seifert fiber spaces.

\begin{proof}[Proof of Theorem \ref{SFSThm}]
To construct the manifold $M$ we proceed as explained in the construction of Seifert fiber spaces with the given parameters at the beginning of this section. Let $D$ be a $2$-disk and let \{$\sigma_i$, $\sigma'_i\}$ be a collection of pairwise disjoint embedded arcs in $\del D$. Let $\phi_i: \sigma_i(I) \to \sigma'_i(I)$ be a homeomorphism that is either $\sigma_i(s) \to \sigma'_i(s)$ for all $s\in I$ or $\sigma_i(s) \to \sigma'_i(1-s)$ for all $s\in I$. Let $\psi_i :S^1 \to S^1$ be either the  identity map $z \to z$  or the conjugation map $z \to \bar{z}$ for all $z\in S^1 \subset \C$. The number of such arcs $\sigma_i, \sigma'_i$ and the choice of $\phi_i$ and $\psi_i$ is determined by the parameters. See Section 2 of \cite{CMMN} for details.

Let $D \times S^1$ be a solid torus foliated by the circle leaves $x \times S^1$. Let $M^*=D \times S^1/\sim$ where $(\sigma_i(s), z) \sim (\phi_i(\sigma_i(s)), \psi_i(z))$, i.e., $M^*$ is obtained from the solid torus $D \times S^1$ by identifying the annuli $\sigma_i(I) \times S^1$ with $\sigma'_i(I) \times S^1$ via the maps $\phi_i \times \psi_i$. As $\phi_i$ and $\psi_i$ send leaves of $D \times S^1$ to leaves of $D \times S^1$ so they induce a foliation of $M^*$ by circle leaves. As $\psi_i(1)=1$ for all $i$, so let $B^* = D \times 1/\sim$ be the surface obtained by identifying the arcs $\sigma_i\times 1$ with $\sigma'_i\times 1$ via the map $\phi_i$. By construction, $B^*$ intersects each leaf exactly once. Let $f: M^*\to B^*$ be the projection map which collapses each circle leaf of $M^*$ to a point. This gives a circle bundle structure on $M^*$.

The manifold $M$ is now obtained from the circle bundle $M^*$ as follows: First Dehn fill $r+1$ torus boundary components $T_i$ to obtain the manifold $M'$, as described in Step 1 of the construction. Then cap off some torus boundary components of $M'$ by $S^1 \times N$ and some Klein bottle boundary components of $M'$ by $S^1 \widetilde{\times} N$ as explained in Step 2 and Step 4 of the construction. And lastly attach copies of $I \times N$ along its boundary to disjoint fibered annuli $I \times S^1$ in some of the boundary components of $M'$, as detailed in Steps 3 and 5 of the construction.\newline \newline
\noindent \emph{Case I: $\del M \neq \emptyset$.}
Let $(p_{r+1}, q_{r+1})=(1,b)$ and let $n=2p_1...p_{r+1}$. Our aim is to construct a horizontal surface $S$ which intersects each regular fiber of $M$ $n$ times. 


\emph{Constructing horizontal surface $S^*$ in $M^*$: }
As $\del M \neq \emptyset$, there exists an arc $d$ in $B^* \cap \del M$. Let $P(n)$ be the set of points $e^{i (m+ 1/2) 2\pi/n}$ in $S^1$ for $m\in \Z$, as in Lemma \ref{wellspaced}. Let $\mathcal{D} = D \times P(n) \subset D \times S^1$. By Lemma \ref{wellspaced}, $\psi_i(P(n))=P(n)$ for all $i$, so the arcs $\sigma_i \times P(n) \subset \del \mathcal{D}$ are identified with the arcs $\sigma'_i \times P(n) \subset \del \mathcal{D}$ to give a horizontal surface in $M^*$ that we denote by $S^*$. The map $f: S^*\to B^*$ is an $n$-to-$1$ covering projection. 

\begin{figure}
\centering
\def\svgwidth{0.6\columnwidth}
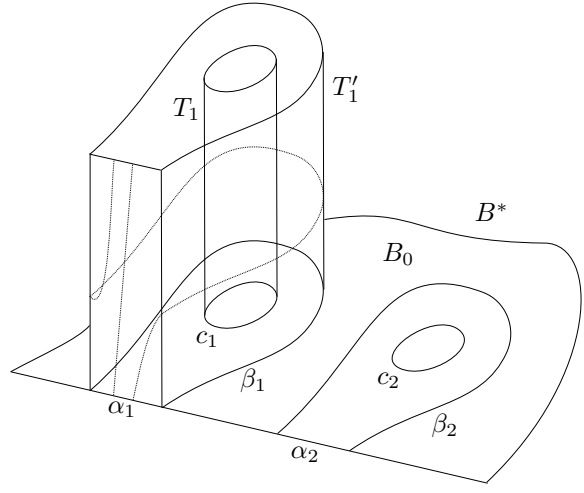
\caption{A piece of the section $B^*$ in $M^*$ where $T_1$ is the torus boundary above $c_1$, $T'_1$ is the parallel torus  above $\alpha_1 \cup \beta_1$ and the dotted curve on $T'$ represents the surgery slope $\frac{1}{2}$.}\label{Dehnfill}
\end{figure}

\emph{Constructing horizontal surface $S'$ in $M'$: }
Let $c_1, ..., c_{r+1}$ be the $r+1$ boundary components of $B^*$ whose preimages $T_j=f^{-1}(c_j)$ are tori which are Dehn filled by solid tori $V_j$. Let $\beta_1, ..., \beta_{r+1}$ be disjoint arcs in $B^*$ with both end points on $d$ which cut out from $d$ disjoint arcs $\alpha_j$. Furthermore, they cut out from $B^*$ annuli with disjoint interiors and with boundary curves $\alpha_j \cup \beta_j$ and $c_j$. Let $B_0$ be the component of $B^*$ outside of all these annuli (i.e, $B_0$ is disjoint from all the $c_j$). Let $M_0=p^{-1}(B_0)$ be its pre-image in $M^*$. Let $T'_j=f^{-1}(\alpha_j \cup \beta_j)$ be a torus parallel to $T_j=f^{-1}(c_j)$. As in Lemma \ref{welltwisted} let $\Lambda_j$ be a set of $n/p_j$ pairwise disjoint curves of slope $q_j/p_j$ in $T'_i$ such that $\Lambda_j\cap f^{-1}(\beta_j)=\beta_j \times P(n)$ and $\Lambda_j \cap f^{-1}(\alpha_j(s))=\alpha_j(s) \times e^{2\pi i s (q_j /p_j)} P(n)$. In Step 1 of our construction, we obtained $M'$ by Dehn filling $M^*$ along the slopes $q_j/p_j$ of $T_j$. We can instead construct $M'$ by attaching the meridians of the solid tori $V_j$  to $T'_j$ along the curves in $\Lambda_j$. See Figure \ref{Dehnfill}.

Let $S_0 = M_0 \cap S^*$ be a horizontal surface in $M_0$ which intersects each fiber $f^{-1}(b)$ of the circle bundle $f: M_0 \to B_0$ in $b \times P(n)$. In particular, $S_0$ intersects each annulus $f^{-1}(\beta_i)$ in $\beta_i \times P(n)$.


Let $\mathcal{D}_i$ be a union of $n/p_j$ disjoint meridian disks in the solid torus $V_j$. To construct $M'$ from $M_0$ we attach $V_j$ to $M_0$ along $f^{-1}(\beta_j)$ by homeomorphisms $h_j$ from $\del V_j$ to $T'_j$ that send $\del \mathcal{D}_j$ to $\Lambda_j$. As $S_0 \cap f^{-1}(\beta_j) = \beta_j \times P(n)=h_j(\mathcal{D}_j) \cap f^{-1}(\beta_j)$ so $S'=S_0 \cup h_j(\mathcal{D}_j)$ is a horizontal surface in $M'$. Let $F$ be a boundary component of $M'$ and let $\eta$ be a fiber of $F$. If $\eta$ is a fiber over a point in some $\alpha_j$ then $S' \cap \eta = e^{i\theta} P(n)$ for some angle $\theta$. By Lemma \ref{wellspaced}, $\omega(e^{i\theta} P(n))=e^{i\theta} \omega(P(n))=e^{i\theta} P(n)$. If $\eta$ is not a fiber over a point in some $\alpha_j$, then $\eta \cap S'=P(n)$. So in either case $\omega(S' \cap \eta) = S' \cap \eta$.

\emph{Extending $S'$ to a horizontal surface $S$ in $M$:}
Let $F$ be a boundary component of $M'$ that is capped off by $S^1 \times N$ (or $S^1 \widetilde{\times} N$) as in Step 2 (or Step 4) of the construction of Seifert fiber spaces with given parameters. Let $\Gamma$ be the disjoint union of curves $\Gamma=S' \cap F$. For each fiber $\eta$ of $F$, $\omega(\Gamma \cap \eta)=\Gamma \cap \eta$ and so by Lemma \ref{extension2}, there exists a horizontal surface $\mathcal{A}_F$ in $S^1 \times N$ (or in $S^1 \widetilde{\times} N$) such that $\mathcal{A}_F \cap F = \Gamma$.

Let $\gamma$ be an arc in $\del B^*$ disjoint from all the $\alpha_i$ and $c_i$. Assume that we need to attach a copy of $I \times \del N$ along $f^{-1}(\gamma)$ as in Step 3 or Step 5 of the construction of Seifert fiber spaces with given parameters. Then $S' \cap f^{-1}(\gamma) = \gamma \times P(n)$. By Lemma \ref{extension}, there exists a horizontal surface $\mathcal{A}_\gamma$ in $I \times N$ such that $\mathcal{A}_\gamma \cap f^{-1}(\gamma) = \gamma \times P(n)$.

We can therefore extend the surface $S'$ to $S= S' \cup_F \mathcal{A}_F \cup_\gamma \mathcal{A}_{\gamma}$ where $F$ varies over all boundary surfaces of $M'$ which are capped off by an $S^1 \times N$ or $S^1 \widetilde{\times} N$ and $\gamma$ varies over all arcs in $\del B^*$ such that an $I \times N$ is attached to $f^{-1}(\gamma)$. By construction if $\eta$ is a fiber of $\del M$ then $\omega(S\cap \eta) = S \cap \eta$.\\

\noindent \emph{Case II: $\del M=\emptyset$ and $SE(M) \neq \emptyset$.}
If $SE(M)$ has an annulus then $\del M \neq \emptyset$, as such annuli can only be obtained as the exceptional set of an $I \times N$ attached to $\del M'$. So we may assume that $SE(M)$ only has torus and Klein bottle components. These are obtained as the exceptional sets of $S^1 \times N$ or $S^1 \widetilde{\times} N$ attached to $M'$ along boundary components. 

Assume that $SE(M)$ has a torus exceptional set obtained by attaching $P = S^1 \times N$ along a torus boundary component $T$ of $M'$. Let $W=M \setminus int(P)$. As $W$ is a Seifert fiber space with boundary $T$ so by Case I, $W$ contains a horizontal surface $S_W$. Furthermore for each fiber $\eta$ of $T$, $\omega(S_W \cap \eta)=S_W \cap \eta$. So by Lemma \ref{extension2}, there exists a horizontal surface $\mathcal{A}_P$ in $P$ such that $\mathcal{A}_P \cap T = S_W \cap T$. Suppose $SE(M)$ has a Klein bottle exceptional set which lies in $Q=S^1 \widetilde{\times} N$ attached to a Klein bottle boundary component $K$ of $M'$. Proceeding similarly, we  get a horizontal surface $\mathcal{A}_Q$ in $Q$ such that $\mathcal{A}_Q \cap K = S_W \cap K$. Therefore either $S=S_W \cup_T \mathcal{A}_P$ or $S=S_W \cup_K \mathcal{A}_Q$ is the required  horizontal surface in $M$.\\

\emph{Case III: Suppose $\del M = \emptyset$ and $SE(M)=\emptyset$.}
The proof here is identical to the closed orientable case (see Pg 26-27 of \cite{Hat}). We reproduce here the details for completion. Remove a solid torus neighbourhood $V$ of a regular fiber of $M$ to get a manifold $W$ with a torus boundary component $T$.  Proceed as in Case I, to obtain a horizontal surface $S_W$. It is now enough to show that $S_W$ intersects $T$ in curves of slope $e(M)=\sum_{i=1}^{r+1} q_i/p_i$: If $e(M)=0$, we can extend the horizontal surface $S_W$ to a horizontal surface on all of $M$ by attaching meridian disks of the solid torus $V$ that we Dehn fill in at $T$ with slope zero. Conversely, given a horizontal surface $S$ in $M$, the intersection of $S$ with $T$ bounds disks in $V$ and hence must have slope zero, so $e(M)=0$.

\emph{Claim: Slope of $S \cap T$ is $e(M)$.} As $S$ is horizontal it meets each fiber of $M_0$ the same number of times, say $n$ times. Intersections of $S$ with $B_0$ on the boundary we count with sign according to whether the slope of $\del S$ at such an intersection point is positive or negative. The signed total number of intersections we get is zero as points at the end of an arc of $S \cap B_0$ have opposite sign. The slope of $S$ on the torus boundary containing $c_i$ is by definition the ratio of the signed intersection with $\del B_0$ and the signed intersection with a regular fiber. As this slope is $q_i/p_i$, it gives the signed intersection of $S$ with $B_0$ on $c_i$ as $n (q_i/p_i)$ for $i=1...(r+1)$. So the slope of $S\cap T$ is $(\sum_{i=1}^{r+1} n q_i/p_i )/n = \sum_{i=1}^r q_i/p_i + b = e(M)$.

\end{proof}

The below Corollary \ref{SFSCor} now follows from standard arguments for horizontal surfaces (see the discussion on Pg 17-18 of \cite{Hat}):

\begin{corollary}\label{SFSCor}
Let $M$ be a compact $3$-manifold and let $F$ be a compact $2$-sided surface properly embedded in $M$. The following are equivalent:
\begin{enumerate}
{\item{$M$ is a Seifert fiber space and $F$ is a horizontal surface in $M$ that intersects each regular fiber of $M$ $n$ times.}
\item{At least one of the following is true:
	\begin{enumerate}
	{\item{There exists a homeomorphism $\phi$ of $F$ such that $M = F \times I/\sim$, where $(x, 1) \sim (\phi(x), 0)$ for all $x\in F$. Furthermore $\phi^n=id$.}
	\item{There exist homeomorphisms $\psi_0$ and $\psi_1$ of $F$ such that $M=F \times I/\sim$, where $(x, 0) \sim (\psi_0(x), 0)$ and $(x, 1) \sim (\psi_1(x), 1)$. Furthermore, $n$ is even, $(\psi_0 \psi_1)^{n/2}=id$ and both $\psi_0$ and $\psi_1$ are fixed-point free involutions.}
	}\end{enumerate}
}
}\end{enumerate}
\end{corollary}

\begin{proof}
Let $M$ be a Seifert fiber space and let $F$ be an embedded $2$-sided horizontal surface in $M$ that intersects each regular fiber $n$ times. 
Each fiber of $M$ has a fibered neighbourhood fiber-preserving homeomorphic to the fibered solid torus $D \times I/\sim$ where $(x, 1) \sim(h(x), 0)$ for some homeomorphism $h$ of $D$. So $M\setminus F$ is an $I$-bundle. Let $p:M\setminus F \to G$ be the $I$-bundle projection, with $G$ the base surface. As $F$ is $2$-sided, the associated $\del I$-subbundle is two copies of $F$. Let $N(F)$ denote the tubular neighbourhood of $F$. Then $p: \del (M\setminus N(F)) = F \sqcup F \to G$ is a $2$-sheeted covering projection.

If $M\setminus F$ is connected then $G$ is connected and the covering map $p: F \cup F \to G$ is the identity on each copy of $F$. So $M\setminus F= F\times I$ and hence $M=F \times I/\phi$ for some homeomorphism $\phi: F\to F$. Let $x\in F$ and let $\eta_x$ be the fiber above $x$. As $F\times \frac{1}{2}$ intersects $\eta_x$ $n$ times, so $\eta_x$ is divided by $F$ into $n$ segments with end points $(\phi^i(x), \frac{1}{2})$ and $(\phi^{i+1}(x), \frac{1}{2})$. So in particular, $\phi^n(x)=x$ as required.

If $M\setminus F$ is disconnected then each of the two components of $M\setminus F$ are $I$-bundles with a copy of $F$ as the associated $\del I$-subbundle. The base surface $G$ has two components $G_0$ and $G_1$ and the projection map restricted to each copy of $F$ is a $2$-sheeted cover $p_i: F \to G_i$ for $i=0,1$. Let $\psi_i: F\times i \to F\times i$ be the non-trivial deck transformation corresponding to $p_i$. As the group of deck transformations is $\Z_2$ so $\psi_i^2=id$. As $F$ is $2$-sided, by thickening $F$ to $F \times I$ and collapsing the two $I$ bundles along the fibers, we get $M= F \times I/\psi_i$ as required. Let $x\in F$ and let $\eta_x$ be the fiber above $x$. As before, $F\times \frac{1}{2}$ intersects $\eta_x$ $n$ times so $\eta_x$ is divided by $F\times \frac{1}{2}$ into $n$ segments along the points $(x, \frac{1}{2})$, $(\psi_1(x),\frac{1}{2})$, $(\psi_0 \psi_1(x),\frac{1}{2})$, $(\psi_1 \psi_0 \psi_1(x), \frac{1}{2})$, ...., $(\psi_1 (\psi_0 \psi_1)^{(n/2)-1}(x), \frac{1}{2})$. In particular, $n$ is even and $(\psi_0 \psi_1)^{n/2}(x)=x$ as required. As $\psi_i$ is a non-trivial deck transformation so it is fixed-point free.

Conversely, let $M=F \times I/\phi$ be a surface bundle with periodic monodromy $\phi$ of period $n$. $F \times I$ is foliated by the leaves $x \times I$ for $x\in F$. For each $x\in F$, let $U_x$ be a neighbourhood of $x$ in $F$ homeomorphic to $\R^2$. Then $U_x \times [0, \frac{1}{2}) \cup (\phi(U_x) \times (\frac{1}{2}, 1])$ is a fibered neighbourhood fiber-wise homeomorphic to $\R^2 \times \R$ with the leaves $x \times \R$. Therefore, the leaves $x \times I/\phi$ give a foliation of $M$ with $1$-dimensional leaves. As $\phi^{n}(x)=x$, so each leaf $\cup_{i=0}^\infty (\phi^i(x) \times I)$ is a circle.

Similarly, assume $M = F \times I/\psi_i$ with $\psi_i$ fixed-point free, $\psi_i^2=id$ and $(\psi_0 \psi_1)^{n/2}=id$. As $\psi_i$ is fixed-point free, for each point $(x, i) \in F \times i$, there exists a neighbourhood $U_{(x, i)} \subset F$ homeomorphic to $\R^2$ such that $\psi_i(U_{(x, i)})\cap U_{(x, i)} = \emptyset$. If this were not true we would obtain a sequence of points $x_n \to x$ in $F$ such that $\psi_i(x_n) \to x$. As $\psi_i$ is an involution so it follows that $x_n \to \psi_i(x)$ and hence $x = \psi_i(x)$. $U_{(x,1)} \times (0, 1] \cup \psi_1(U_{(x, 1)}) \times (0, 1]$ is then a fibered neighbourhood fiber-wise homeomorphic to $\R^2 \times \R$ (with the leaves $x \times \R$). And similarly $U_{(x,0)} \times [0, 1) \cup \psi_0(U_{(x, 0)}) \times [0, 1)$ is a fibered neighbourhood fiber-wise homeomorphic to $\R^2 \times \R$ with the leaves $x \times \R$. $M$ is therefore foliated by the circle leaves $(x \times I) \cup (\psi_1(x) \times I) \cup (\psi_0 \psi_1(x) \times I) \cup... \cup (\psi_1(\psi_0\psi_1)^{(n/2)-1} \times I)$.
 
We now use a result of Epstein\cite{Eps} which says that that any compact 3-manifold foliated by circles is a Seifert fiber space to conclude that $M$ is Seifert fibered. Furthermore, it contains the $2$-sided horizontal surface $F \times \frac{1}{2}$ which intersects each regular fiber $n$ times. By Theorem \ref{SFSThm}, it must have $\del M \neq \emptyset$ or $SE(M) \neq \emptyset$ or be closed with $e(M)=0$.
\end{proof}

\section{Prism complexes}
Let $F$ be a surface with a Riemannian metric $g$. We begin this section with an overview of the existence and uniqueness of a Riemannian center of mass for small enough convex geodesic polyhedra in $(F, g)$. The Euclidean center of mass of points $p_1, ..., p_k \in \R^n$ is the point $\frac{1}{k}\sum p_i$. The Riemannian center of mass, also known as the Karcher mean, is a generalisation of this affine notion and was extensively studied by Karcher\cite{Kar}. We present here the treatment as in \cite{DyeVegWin}.

\begin{definition}
For $x\in F$, let $B(x, r)$ denote the set of points of $F$ at a distance less than $r$ from $x$, and denote by $\overline{B(x,r)}$ its closure. The injectivity radius of $M$ at a point $x\in M$ is the supremum of the radii $r$ of Euclidean balls $B(0, r) \subset T_x(M)$ that project down diffeomorphically to balls $B(x, r)$ in $M$ via the exponential map $exp_x$. The injectivity radius of $M$, denoted by $i(M)$, is the infimum of the injectivity radius at all points of $M$.
\end{definition}

We call a set $B \subset M$ convex if any two points $p, q \in B$ are connected by a minimising geodesic that is unique in $M$ and which lies entirely in $B$.

\begin{lemma}[Theorem IX.6.1 of \cite{Cha}]
Let $M$ be a Riemannian manifold with sectional curvatures bounded above by $K_+$ and let $i(M)$ be its injectivity radius. If 
$$r < \min\left\{\frac{i(M)}{2}, \frac{\pi}{2\sqrt{K_+}}\right\}$$
then $\overline{B(x, r)}$ is a convex set. (If $K_+ \leq 0$ then we take $1/\sqrt{K_+}$ to be infinite.)
\end{lemma}

Let $B$ be an open set in $M$ such that $\overline{B}$ is convex. Let $P\subset B$ be a geodesic convex polyhedron with vertices $\{p_1, ... p_k\}$. Let $d$ denote the Riemannian distance function in $M$. Let $\epsilon: \overline{B} \to \R$ be the smooth function
$$ \epsilon(x) = \frac{1}{2k} \sum_{i=1}^k d(x, p_i)^2$$
The gradient of $\epsilon$ is given by
$$ \mbox{grad}(\epsilon)(x)=-\frac{1}{k} \sum_i exp_x^{-1}(p_i)$$
At any point $x\in \del P$, this gradient is therefore a vector pointing outward from $P$. And hence a minimum of $\epsilon$ lies in the interior of $P$. Karcher proved that when $B$ is small enough, $\epsilon$ is convex and hence this minimum is unique. He in fact proved this in more generality for sets of measure 1 (as opposed to a set of $k$ points with point measure $1/k$) and with an explicit bound on the convexity of $\epsilon$. The following lemma follows from Theorem 1.2 of \cite{Kar} (see also Lemma 3 of \cite{DyeVegWin}):

\begin{lemma}\label{center}
If $\{p_1, ..., p_k\} \subset B(x, r) \subset M$ with 
$$r< \rho=\min\left\{\frac{i(M)}{2}, \frac{\pi}{4\sqrt{K_+}}\right\}$$
Then the function $\epsilon$ has a unique minimum in $B(x, r)$.
\end{lemma}

\begin{definition}
Given a convex geodesic polyhedron $P \subset B(x, \rho) \subset M$ with vertices $\{p_1, ..., p_k\}$, we call this unique minimum $b(P)$ of $\epsilon$ in $P$ the barycenter of $P$. 
\end{definition}

\begin{lemma}\label{isomcenter}
Let $\phi$ be an isometry of $(F, g)$, let $P\subset B(x_0, \rho)$ be a convex geodesic polyhedron and let $Q=\phi(P)$. Then $\phi(b(P))=b(Q)$.
\end{lemma}
\begin{proof}
Let $V(P)=\{p_1, ..., p_k\}$ and $V(Q)=\{q_1, ..., q_k\}$ be the set of vertices of $P$ and $Q$ respectively. Then $\phi(V(P))=V(Q)$ and so for all $x\in P$, 
$$\epsilon_Q (\phi (x)) =  \frac{1}{2k} \sum_{i=1}^k d(\phi(x), q_i)^2 = \frac{1}{2k} \sum_{i=1}^k d(\phi(x), \phi(p_i))^2$$
As $\phi$ is an isometry, so it follows that $\epsilon_Q \circ \phi = \epsilon_P$. Let $\epsilon_P(b(P))=m$ be the minimum value of $\epsilon_P$. So $\epsilon_Q(\phi(b(P)))=\epsilon_P(b(P))=m$ and as $\epsilon_Q=\epsilon_P \circ \phi^{-1}$ so the minimum value of $\epsilon_Q$ is also $m$. As $b(Q)$ is the unique minima of $\epsilon_Q$ so  $\phi(b(P))=b(Q)$.
\end{proof}

We show below that a periodic surface automorphism is simplicial with respect to some triangulation.

\begin{lemma}\label{periodicsim}
Let $H$ be a finite subgroup of the group of automorphisms of a compact surface $F$. There exists a triangulation $\tau$ of $F$ such that each $\phi\in H$ is a simplicial map with respect to $\tau$. 
\end{lemma} 
\begin{proof}
Let $n$ be the order of $H$. Let $g_0$ be a Riemannian metric on $F$ and let $g=\sum_{h\in H} h^*(g_0)$. Any $\varphi \in H$ acts on $H$ as $\varphi(h)=h \circ \varphi$ for all $h\in H$ to give a bijection of $H$. So $\varphi^*g = \sum_{h\in H} (h \circ \varphi)^*(g_0) = g$, i.e., $\varphi$ is an isometry of $(F, g)$. 

Let $\tau_0$ be a geodesic triangulation of $F$ such that each simplex lies in a convex ball of radius less than $\rho$ as defined in Lemma \ref{center}. Let $\Pi$ denote the polyhedral complex obtained by intersecting the simplexes of $h(\tau_0)$ for $h\in H$. In other words, if $H=\{h_1, ..., h_n\}$ then each cell $P$ of $\Pi$ is obtained by taking $n$ triangles $\Delta_1$, ..., $\Delta_n$ in $\tau_0$ (possibly with repetition) and taking the intersection $P=\cap_{i=1}^n h_i (\Delta_i)$. As each $h_i(\Delta_i)$ is convex so $P$ is a convex polyhedron. As $\varphi$ induces a permutation of $H$ so $\varphi$ is a map sending polyhedra of $\Pi$ to polyhedra.


For each polyhedron $P$ of $\Pi$ let $V(P)$ be its set of vertices and let $b(P)\in int(P)$ denote its Riemannian center of mass. $P$ can be subdivided into the triangulation $\tau_P = b(P) \star \del P$ by dividing along edges joining $b(P)$ to the vertices of $P$. Let $\tau$ be the triangulation obtained by replacing each polyhedron $P \in \Pi$ with the triangulated polyhedron $\tau_P$. 

For any $\varphi \in H$ as $\varphi$ is a polyhedral map on $\Pi$, so if $P$ is a polyhedron in $\Pi$ so is $Q=\varphi(P)$. As $\varphi(V(P))=V(Q)$ and by Lemma \ref{isomcenter} $\varphi(b(P))=b(Q)$, so $\varphi$ is in fact a simplicial map from $\tau_P$ to $\tau_Q$. Hence $\varphi$ is simplicial over $\tau$ as required.
\end{proof}

\begin{figure}
\centering
\def\svgwidth{0.8\columnwidth}
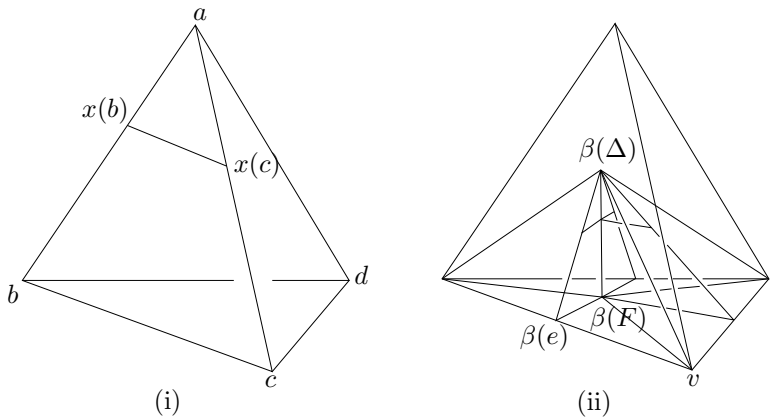
\caption{(i) Converting a 3-simplex into a prism (ii) Consistently doing so in the barycentric subdivision of a 3-simplex}\label{baryprism}
\end{figure}

We now prove the main theorem of this article:
 \begin{proof}[Proof of \ref{mainthm}]
Given a 3-simplex $\Delta$ with vertices $a, b, c, d$ we can convert it to a prism that we denote as $[a; b, c; d]$ by introducing a vertex $x(b)$ in the edge $[a,b]$ a vertex $x(c)$ in the edge $[a,c]$ and an edge $[x(b), x(c)$] that divides the face $[a,b,c]$ into a triangle which contains $a$ and a quadrilateral which contains $b$ and $c$, as shown in Figure \ref{baryprism}(i). This gives a prism with triangular faces $[a,x(b),x(c)]$ and $[b,c,d]$ and quadrilateral faces $[b, c, x(c), x(b)]$, $[c, d, a, x(c)]$ and $[d, b, x(b), a]$.

 To change tetrahedra to prisms consistently, we work instead with the barycentric subdivision $\beta(\tau)$ of a simplicial triangulation $\tau$ of $M$. Let $\beta(\sigma)$ denote the barycenter of a simplex $\sigma$. Any 3-simplex in $\beta(\tau)$ is of the form $[\beta(\Delta), \beta(F), \beta(e), v]$ where $\Delta$ is a 3-simplex of $\tau$, $F$ a 2-simplex of $\Delta$, $e$ an edge of $F$ and $v$ a vertex of $e$. To obtain a prism complex structure we change each such simplex to the prism $[\beta(\Delta); \beta(F), \beta(e); v]$, by introducing a vertex $x(F)$ on the edge $[\beta(\Delta), \beta(F)]$, a vertex $x(e)$ on the edge $[\beta(\Delta), \beta(e)]$ and by splitting the face $[\beta(\Delta), \beta(F), \beta(e)]$ along an edge $[x(F), x(e)]$. See Figure \ref{baryprism}(ii) for such a construction on the 3-simplexes of $\beta(\tau)$ in $\Delta$ which contain the 2-simplexes of $\beta(F)$, for a fixed face $F$ of $\Delta$. Such a change on each 3-simplex of $\beta(\tau)$ is consistent. Also any horizontal face is either of the form $[\beta(\Delta), x(F), x(e)]$ and lies in the interior of a 3-simplex of $\tau$ or is of the form $[\beta(F), \beta(e), v]$. Varying $\Delta$, $F\in \Delta$ and $e\in F$, the union of the faces $[\beta(F), \beta(e), v]$ gives the barycentric subdivision of the 2-skeleton of $\tau$. In either case, horizontal edges only meet other horizontal edges, so this construction transforms a simplicial complex to a prism complex. As every compact 3-manifold has a simplicial complex structure it therefore has a prism complex structure. This construction does not however give a special prism complex structure, in particular, the interior horizontal edge $[x(F), x(e)]$ lies in only two prisms.

Suppose $M$ admits a special prism complex structure. Foliate each prism $\Delta \times I$ by intervals $x \times I$. Each face in the interior of the complex is shared by exactly two prisms while each boundary face lies in one prism. Furthermore as horizontal edges are identified only with horizontal edges, so horizontal faces are identified only with horizontal faces. Therefore points in the interior of faces have fibered neighborhoods that are fiber-wise homeomorphic to the fibered product $D\times I$ if the point is in the interior of $M$ and $D^+ \times I$ if the point is on the boundary of $M$. The star of an edge is the union of all prisms which contain the edge. The dual graph of the star of an interior edge of the complex is regular of degree 2 and is therefore a circuit. So points in the interior of vertical edges have neighborhoods fiber-wise homeomorphic to $D\times I$. Similarly points on a vertical edge on the boundary has neighbourhoods fiber-wise homeomorphic to $D^+ \times I$. Exactly 4 prisms meet at a horizontal edge so exactly 2 horizontal faces meet along a horizontal edge. Consequently, the union of all horizontal faces gives a triangulated surface $S$. The interior of the union of all triangles containing a vertex $v$ in $S$ is a disk. All the prisms with a horizontal face on this disk which lie on the same side of the disk, share the vertical edge containing $v$. And so the union of all prisms containing such a vertex in $M$ contains a fibered neighbourhood of $v$ fiber-wise homeomorphic to $D \times I$. Therefore $M$ is foliated by 1-dimensional leaves. 

As no horizontal face lies on the boundary so the dual graph of the prism complex with edges corresponding to horizontal faces and vertices corresponding to the prisms is also a circuit. The union of the corresponding prisms is then either a solid torus or a solid Klein bottle foliated by circles. This shows that the 1-dimensional foliation constructed above has only circle leaves. Epstein\cite{Eps} has shown that any compact 3-manifold foliated by circles is a Seifert fiber space. As the surface $S$ consisting of horizontal faces is transverse to this foliation, so by Theorem \ref{SFSThm} either $\del M\neq \emptyset$, $SE(M)\neq\emptyset$ or $M$ is closed with $e(M)=0$  as required.

Conversely, if $M$ is a Seifert fiber space with $\del M\neq\emptyset$, $SE(M)\neq\emptyset$ or $e(M)=0$ then by Corollary \ref{SFSCor}, $M= F\times I/\phi$ where $\phi:F\times\{1\}\rightarrow F\times \{0\}$ is a periodic monodromy or $M= F\times I/\psi_i$ i=0,1 where $\psi_i: F\times \{i\}\rightarrow F\times \{i\}$ is an involution. 

If $M=F \times I/\phi$ then let $H$ be the finite subgroup of $Aut(F)$ generated by $\phi$. By Lemma \ref{periodicsim}, there exists a triangulation $\tau$ of $F$ with respect to which $\phi$ is simplicial. Let $\tau \times I$ be a prism complex structure on $F \times I$ where each prism is of the form $\Delta \times I$ for $\Delta$ a triangle of $\tau$. Then $\Pi = \tau \times I/\phi$ is the required prism triangulation $M$. Similarly, if $M=F \times I/\psi_i$ then let $H$ be the finite subgroup of $Aut(F)$ generated by $\psi_1$ and $\psi_2$. By Lemma \ref{periodicsim}, there exists a triangulation $\tau$ of $F$ with respect to which both $\psi_1$ and $\psi_2$ are simplicial. Then $\Pi=\tau \times I/\psi_i$ is the required prism triangulation of $M$.
 \end{proof}

\begin{acknowledgements}
The first author was supported by the MATRICS grant of Science and Engineering Research Board, GoI.
\end{acknowledgements}

\bibliographystyle{alpha}
\bibliography{Prism}









\end{document}